\documentclass[11pt,reqno]{amsart}
\usepackage{latexsym}
\RequirePackage{amsmath}   \RequirePackage{amssymb}
\RequirePackage{amsthm}   \RequirePackage{epsfig}
\makeatletter
\@namedef{subjclassname@2020}{\textup{2020} Mathematics Subject Classification}
\makeatother

\theoremstyle{plain}
\newtheorem{theorem}{Theorem}

\newtheorem{lemma}{Lemma} %[theorem]
\newtheorem{proposition}{Proposition} %[theorem]
\newtheorem*{corollary*}{Corollary}

\newtheorem{conjecture}{Conjecture}
\newtheorem*{conjecture*}{Conjecture}

\theoremstyle{definition}

\theoremstyle{remark}
\newtheorem{remark}{Remark}

\newcommand{\SC}{{\mathbb C}}  \newcommand{\SD}{{\mathbb D}} 
 \newcommand {\SR}{{\mathbb R}}    

\newcommand{\al}{\alpha}    
  \newcommand{\ve}{\varepsilon}  
\def\t{\theta}    \newcommand{\la}{\lambda}
\newcommand{\si}{\sigma}  \newcommand{\vp}{\varphi}  \newcommand{\om}{\omega}
  \newcommand{\Om}{\Omega}

\newcommand{\be}{\begin{equation}}
\newcommand{\ee}{\end{equation}}
\newcommand{\bea}{\begin{eqnarray}}
\newcommand{\eea}{\end{eqnarray}}

\begin{document}

\title{Harmonic mappings, univalence criteria and a theorem of Lehtinen}  

\author[I. Efraimidis]{Iason Efraimidis}
\address{Departamento de Matem\'aticas, Universidad Aut\'onoma de Madrid, 28049 Madrid, Spain} \email{iason.efraimidis@uam.es}

\author[R.~Hern\'andez]{Rodrigo Hern\'andez}
\address{Facultad de Ingenier\'{\i}a y Ciencias, Universidad Adolfo Ib\'a\~nez, Av. Padre Hurtado 750, Vi\~na del Mar, Chile} \email{rodrigo.hernandez@uai.cl}

\subjclass[2020]{30C99, 30G30} 
\keywords{Harmonic mappings, Schwarzian derivative, univalence criteria}

\maketitle

\begin{abstract}
The harmonic inner radius $\sigma_H(\Omega)$ of a planar domain $\Omega$ is the largest constant with which a univalence criterion via the Schwarzian derivative holds for harmonic mappings. We show that $\sigma_H(\Omega)\leq\sigma_H(\mathbb{D})\leq 3/2$ for the unit disk $\mathbb{D}$ and for every domain $\Omega$ that omits an open set. This is an analogue of a theorem of Lehtinen in the setting of holomorphic functions. We provide two related univalence criteria for harmonic mappings.
\end{abstract}

\vskip.6cm

\begin{center}
\textit{Dedicated to Professor S.~Ponnusamy on the occasion of his 65th birthday.}
\end{center}

\vskip.8cm
%%%%%%%%%%%%%%%%%%%%%%%%%%%%%%%%%%
\noindent \textbf{1.~Introduction.} Let $\Om$ be a hyperbolic domain in $\overline{\SC}$, \emph{i.e.}, a domain that has at least three boundary points, and let $\pi:\SD\to\Om$ be a universal covering map, where $\SD$ is the unit disk. The density $\la_\Om$ of the hyperbolic (Poincar\'e) metric in $\Om$ is defined by 
$$
\la_\Om\big(\pi(z)\big) |\pi'(z)| \, = \, \la_\SD(z)\, = \,\frac{1}{1-|z|^2}, \qquad z\in \SD, 
$$
which is independent of the choice of the covering map $\pi$. For a locally univalent analytic function $f$ in $\Om$ the pre-Schwarzian and Schwarzian derivatives of $f$ are defined by 
$$
Pf = (\log f')' = \frac{f''}{f'} \qquad \text{and} \qquad Sf = (Pf)' -\tfrac{1}{2}(Pf)^2,
$$
respectively, while 
\be \label{def-Schw-norm}
\|Sf\|_\Om \, = \, \sup_{z\in\Om} \, \frac{ |Sf(z)| }{\la_\Om(z)^2} 
\ee
denotes its Schwarzian norm. The inner radius of $\Om$ is defined as the number 
$$
\si(\Om) \, = \, \sup \{\, c \geq0 \, : \, \, \|Sf\|_\Om \leq c \,\Rightarrow \, f \, \text{univalent} \}.   
$$ 
This domain constant is M\"obius invariant, \emph{i.e.}, $\si(T\Om)=\si(\Om)$ for every M\"obius transformation $T$. The classical results of Nehari and Hille show that $\si(\SD)=2$, while those of Ahlfors and Gehring show that for $\Om$ simply connected we have that $\si(\Om) > 0$ if and only if $\Om$ is a quasidisk, \emph{i.e.}, the image of $\SD$ under a quasiconformal self-map of $\overline{\SC}$. Lehtinen \cite{Lhi80} showed that every (hyperbolic) simply connected domain $\Om$ satisfies $\si(\Om) \leq 2$, with equality only in the case when $\Om$ is a disk or a half-plane. For finitely connected domains $\Om$ we have that $\si(\Om) > 0$ if and only if all boundary components are either points or quasicircles (see \cite{BG80, O80}), while uniform domains $\Om$ satisfy $\si(\Om) > 0$ (see \cite{MS79}). For more information on the inner radius we direct the reader to Lehto's book \cite[Ch.III, \S 5]{Leh}.

More generally, for a complex-valued harmonic mapping $f$ in a planar domain $\Om$ we write $\om=\overline{f_{\overline{z}}} / f_z$ for its dilatation. According to Lewy's theorem the mapping $f$ is locally univalent if and only if its Jacobian $J_f=|f_z|^2-|f_{\overline{z}}|^2$ does not vanish. Duren's book \cite{Du} contains valuable information about the theory of planar harmonic mappings. 

The pre-Schwarzian and Schwarzian derivatives of a locally univalent harmonic mapping $f$ were defined in \cite{HM15} as 
$$ 
P_f = (\log J_f)_z  \qquad \text{and} \qquad S_f \, = \, (P_f)_z - \tfrac{1}{2} (P_f)^2. 
$$  
If $f=h+\overline{g}$, for $h$ and $g$ holomorphic, then these formulas are equivalent to 
\be \label{pre-Sch-HM}
P_f \, = \, Ph - \frac{\overline{\om}\om'}{1-|\om|^2}
\ee
and
\be \label{Sch-HM}
S_f \, = \, Sh + \frac{\overline{\om}}{1-|\om|^2} \left( \frac{h''}{h'} \om'-\om'' \right) - \frac{3}{2} \left( \frac{\om' \overline{\om}}{1-|\om|^2} \right)^2. 
\ee
When $f$ is holomorphic these reduce to the classical pre-Schwarzian and Schwarzian derivatives. Observe that the notations $Pf$ and $Sf$ are used when we know that $f$ is holomorphic, while the notations $P_f$ and $S_f$, with $f$ as a subscript, are used in the more general context of harmonic mappings. The Schwarzian norm of a harmonic mapping is defined exactly as in \eqref{def-Schw-norm}.

Another definition for the Schwarzian derivative, introduced by Chuaqui, Duren and Osgood \cite{CDO03}, applies to harmonic mappings which admit a lift to a minimal surface via the Weierstrass-Enneper formulas. Focusing on the planar theory, in this article we adopt the definition from \cite{HM15} given above.

The \emph{harmonic inner radius} of a hyperbolic domain $\Om$ in $\overline{\SC}$ is defined as the constant 
$$
\si_H(\Om) \, = \, \sup \{ \, c \geq0 \, : \, f \; \text{harmonic in}\; \Om \; \text{with}  \; \|S_f\|_\Om\leq c \, \Rightarrow \, f \, \text{injective} \}. 
$$
Clearly, this is a M\"obius invariant constant and it satisfies $\si_H(\Om) \leq \si(\Om)$ since every holomorphic function is harmonic. In \cite{HM15-2} it was shown that $\si_H(\SD)>0$. For finitely connected domains $\Om$ we have that $\si_H(\Om)>0$ if and only if $\si(\Om)>0$; see \cite{Ef1}. Moreover, it was shown in \cite{Ef2} that $\si_H(\Om)>0$ for uniform domains $\Om$.

\vskip.5cm
%%%%%%%%%%%%%%%%%%%%%%%%%%%%%%%%%%
\noindent \textbf{2.~An analogue of Lehtinen's theorem.} Our main result in this section is the following. 

\begin{theorem} \label{thm-main}
For every domain $\Om$ in $\overline{\SC}$ that omits an open set we have that $$\si_H(\Om)\leq \si_H(\SD)\leq 3/2.$$ 
\end{theorem}

%%%%%%%%%%%%%%%%%
\begin{proof}
We first prove the inequality $\si_H(\Om)\leq \si_H(\SD)$. Assume, in order to get a contradiction, that $\si_H(\Om) > \si_H(\SD)$. Then there exists a non-injective harmonic mapping $f$ in $\SD$ for which 
$$
\si_H(\SD) < \|S_f\|_\SD < \si_H(\Om). 
$$
Let $z_1,z_2\in\SD$ be distinct points for which $f(z_1)=f(z_2)$. Let $\vp$ be a disk automorphism for which $\vp(z_1)=0$ and $\vp(z_2)=x \in(0,1)$. % This will be used later. 

Since $\Om$ omits an open set there exists a circle that lies in its complement and touches its boundary. Applying an inversion with respect to this circle we may assume that $\Om\subset\SD$ and that $1\in\partial \Om$ without changing its harmonic inner radius. Moreover, we may assume that $0\in\Om$ since, otherwise, we may apply the transformation 
$$
z\mapsto \frac{1-\overline{p}}{1-p} \, \frac{z-p}{1-\overline{p}z}
$$
for any point $p\in\Om$. Let $\ve = \frac{1-x}{2}$ and consider some point $z_0 = r_0 e^{i\t_0} \in\Om$ for which $|1-z_0|<\ve$. Since $1-r_0\leq|1-z_0|<\ve$ it is evident that 
$$
r_0 > 1-\ve = \frac{1+x}{2} > x.
$$
Hence, the function $\psi(z) = e^{-i\t_0}\frac{x}{r_0} z$ maps $\SD$ into $\SD$, fixes the origin and sends $z_0$ to $x$. Therefore, the domain 
$$
\widehat{\Om} = \vp^{-1} \circ \psi (\Om)
$$
contains the points $z_1$ and $z_2$, so that $f$ restricted on $\widehat{\Om}$ is non-injective. However, by the domain monotonicity of the hyperbolic metric we have that
$$
\|S_f\|_{\widehat{\Om}} \leq \|S_f\|_\SD < \si_H(\Om) = \si_H(\widehat{\Om}), 
$$
which implies that $f$ is injective in $\widehat{\Om}$, a contradiction. 

\vskip.1cm
Next, we prove the inequality $\si_H(\SD)\leq 3/2$. In \cite{WLRS18} it was shown that the harmonic mapping $f_\al=h_\al+\overline{g_\al}$, defined by 
$$
h_\al(z)=\frac{1-(1-z)^\al}{\al} \qquad \text{and} \qquad g_\al(z) = \frac{1-(1+\al z)(1-z)^\al}{\al(1+\al)}, \qquad z\in \SD, 
$$
is not injective for $\al>1$. The dilatation of $f_\al$ is $\om_\al(z)=z$. We compute $h_\al'(z)=(1-z)^{\al-1}$, so that 
$$
Ph_\al(z) =\frac{1-\al}{1-z} \qquad \text{and} \qquad Sh_\al(z)=\frac{1-\al^2}{2(1-z)^2}.
$$ 
Hence, 
$$
S_{f_\al}(z) = \frac{1-\al^2}{2(1-z)^2} + \frac{(1-\al)\overline{z}}{(1-|z|^2)(1-z)} - \frac{3\overline{z}^2}{2(1-|z|^2)^2}. 
$$
It is easy to see that
$$
\|S_{f_\al}\|_\SD \leq 2(\al^2-1) + 2(\al-1) + \frac{3}{2}, 
$$
which decreases to $3/2$ when $\al\to1^+$. Now, if $\si_H(\SD)>3/2$ were true, then there would exist $\al_0>1$ for which the mapping $f_{\al_0}$ would satisfy 
$$
\|S_{f_{\al_0}}\|_\SD \leq 2(\al_0^2-1) + 2(\al_0-1) + \frac{3}{2} = \si_H(\SD).  
$$
But this would imply that $f_{\al_0}$ is injective, a contradiction. 
\end{proof}
%%%%%%%%%%%%%%%%%

A few remarks are in order. 

\begin{remark}
The method of proof of the first inequality in Theorem~\ref{thm-main} can be applied to the holomorphic setting and give a new proof of Lehtinen's inequality $\si(\Om)\leq 2$ that does not rely on any specific function; recall that Lehtinen's proof makes use of the logarithm. On the other hand, an important drawback of our proof is that it fails to characterize the extremal functions, which is the essence of Lehtinen's result. 
\end{remark}

\begin{remark} 
If the domain $\Om$ is simply connected and does not omit an open set then it is not a Jordan domain, hence $\si(\Om)=0$ by a theorem of Gehring \cite{Ge77} and therefore $\si_H(\Om)=0$, so that our inequality holds for \emph{all} simply connected domains. 

The situation is different if $\Om$ is finitely connected and does not omit an open set. In this case, if $\partial \Om$ has some component that is different than a point then $\si(\Om)=0$ by \cite{BG80}. If, on the other hand, $\Om$ is the sphere $\overline{\SC}$ with a finite number of punctures then it is not hard to see that filling in the punctures does not decrease the inner radius, so that $0<\si(\Om) \leq \si(\SC_{0,1})$, where $\SC_{0,1}=\SC\backslash\{0,1\}$ is the largest hyperbolic domain. It seems to be an open problem to estimate the inner radius of $\SC_{0,1}$ or, at least, to show that it satisfies Lehtinen's inequality. 
\end{remark}

\begin{figure}[h]
    \centering
    \begin{minipage}{0.49\textwidth}
        \centering
        \includegraphics[width=0.97\textwidth]{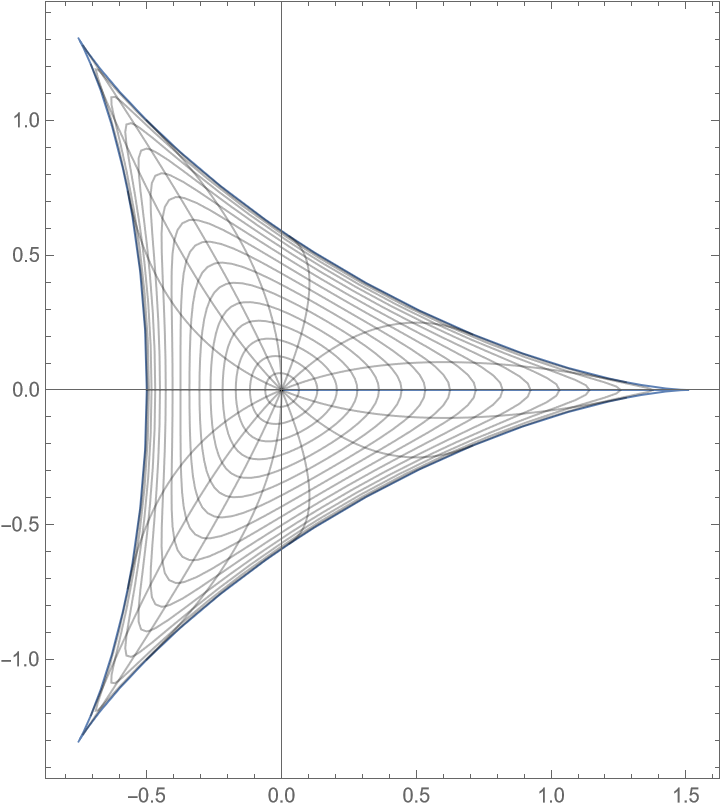}
        \caption{The range of $f_1$} \label{fig1}
    \end{minipage}\hfill
    \begin{minipage}{0.49\textwidth}
        \centering
        \includegraphics[width=0.97\textwidth]{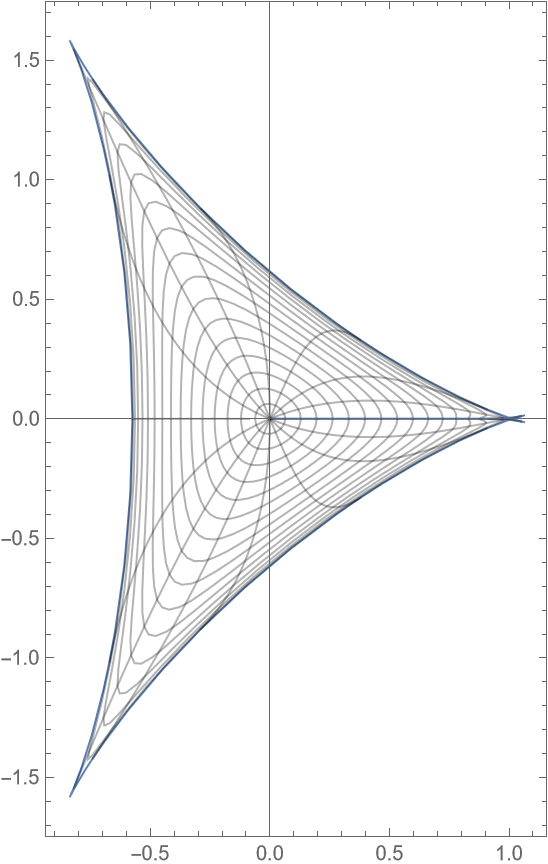} 
        \caption{The range of $f_{1.5}$}\label{fig2}
    \end{minipage}
\end{figure}

\vskip.5cm
%%%%%%%%%%%%%%%%%%%%%%%%%%%%%%%%%%
\noindent \textbf{3.~Two conjectures.} We turn our attention to the unit disk. Hereafter, we simplify the notation for the Schwarzian norm and simply write $\|\cdot\|$ instead of $\|\cdot\|_\SD$. 

Motivated by the harmonic mapping $f_\al$, constructed in \cite{WLRS18} and used in the proof of the second inequality in Theorem~\ref{thm-main}, we state the following conjecture. Note that for $\al =1$, the injective mapping $f_1$ is the horizontal shear of the function $z\mapsto z-z^2/2$ with dilatation $\om(z)=z$; its range $f_1(\SD)$ is shown in Figure~\ref{fig1}. 

\begin{conjecture} \label{conj-3/2}
$\si_H(\SD) = 3/2$ 
\end{conjecture}

According to a theorem of Gehring and Pommerenke \cite{GePo}, if $f$ is holomorphic with $\|Sf\| \leq 2$ then we have that $f(\SD)$ is a Jordan domain and $f$ admits a homeomorphic extension to $\overline{\SC}$, unless $f$ is M\"obius conjugate to the logarithm function $\ell(z)=\log\frac{1+z}{1-z}$, that is, $f=T\circ\ell\circ\tau$ for M\"obius transformations $T$ and $\tau$ with $\tau(\SD)=\SD$. In short, at the threshold $\|Sf\| = 2$ the image is not, in general, a Jordan domain, so that univalence is lost above that threshold. 

The proposed extremal harmonic mapping $f_1$ for Conjecture~\ref{conj-3/2} suggests an interesting distinction between the holomorphic and the harmonic settings: above the threshold $\|S_f\| = 3/2$ injectivity should, in general, be lost because a cusp in the image collapses. Indeed, we see this in the range of $f_{1.5}$ in Figure~\ref{fig2}. One may then expect, in contrast to the holomorphic setting, to still have a Jordan domain at the threshold for harmonic mappings, without the need to exclude any specific mappings from consideration. We are thus led to the following statement, which is stronger than Conjecture~\ref{conj-3/2}. 

\begin{conjecture}
If $f$ is a harmonic mapping in $\SD$ with $\|S_f\|\leq 3/2$ then $f$ admits a homeomorphic extension to $\overline{\SC}$. 
\end{conjecture}

In \cite{Ef1} it was shown that there exists a constant $c>0$ for which the condition $\|S_f\|\leq c$ implies the existence of a homeomorphic extension of $f$ to $\overline{\SC}$.

\vskip.5cm
%%%%%%%%%%%%%%%%%%%%%%%%%%%%%%%%%%
\noindent \textbf{4.~Univalence criteria.} The fact that $\si_H(\SD)>0$, namely, that there exists some $c>0$ such that $\|S_f\| \leq c $ implies that $f$ is injective, raises the pressing question of finding an explicit positive lower bound for $\si_H(\SD)$. Such a bound would be an interesting first step towards solving Conjecture~\ref{conj-3/2}. In what follows, we provide a criterion of a different nature that, still, involves the Schwarzian derivative and, moreover, provides explicit constants. This makes use of the hyperbolic derivative of holomorphic self-maps $\om$ of $\SD$, given by
$$
\om^*(z) = \frac{(1-|z|^2)\om'(z)}{1-|\om(z)|^2}, \qquad z\in\SD,
$$
and the harmonic order operator given by 
\be \label{h-order}
A_f(z) = \tfrac{1}{2} (1-|z|^2) P_f(z) - \overline{z}.  
\ee
It is known that $A_f$ appears naturally as the second coefficient of the analytic part of the harmonic mapping that arises from $f$ after an application of a Koebe transform followed by a certain affine transform; see Section 3 in \cite{AHS22}. 

We will be needing the following lemma. 

\begin{lemma}\label{lem-om}
Let $\om$ be a holomorphic self-mapping of $\SD$. Then 
$$
\left| \frac{(1-|z|^2)^2\om''(z)}{ 2(1-|\om(z)|^2) } - \overline{z} \om^*(z) + \overline{\om(z)} \om^*(z)^2 \right| \leq 1- |\om^*(z) |^2, \qquad z\in\SD.  
$$
\end{lemma}
\begin{proof}
With the notation $\vp_a(z)= \frac{a+z}{1+\overline{a}z}, a\in\SD$, for the automorphisms of $\SD$, the lemma follows directly from considering the self-mapping
$$
\vp_{-\om(a)} \circ \om \circ \vp_a(z) = \frac{ \om\!\left(\frac{a+z}{1+\overline{a}z}\right) - \om(a)}{ 1- \overline{\om(a)} \, \om\!\left(\frac{a+z}{1+\overline{a}z}\right)} = a_1 z + a_2 z^2 +\ldots 
$$
and applying to it the Schwarz-Pick inequality $|a_2|\leq 1-|a_1|^2$. 
\end{proof}

\begin{theorem}\label{thm-h-order}
Let $f$ be a locally univalent harmonic mapping in $\SD$ with dilatation $\om:\SD\to\SD$. If
$$
(1-|z|^2)^2 |S_f(z)| +2| \om^*(z) A_f(z)| \leq \frac{1}{2} | \om^*(z)|^2, \qquad z\in\SD,
$$
then $f$ is injective.
\end{theorem}
\begin{proof}
Writing $f=h+\overline{g}$, we will prove that $\|Sh\|\leq 2$, from which the univalence of $h$ will follow. Thereafter, we observe that the hypothesis is invariant under affine transformations $F=f+a\overline{f}, a\in\SD$, since $P_F=P_f$ and $S_F=S_f$ (see \cite[Prop.~1]{HM15}), so that $A_F=A_f$ and, moreover, the dilatation $\om_F$ of $F$ satisfies
$$
\om_F^* = \frac{1+\overline{a \om}}{1+a\om}\, \om^*, 
$$
therefore $|\om_F^*|=|\om^*|$. It follows that $h+ag$, \emph{i.e.}, the analytic part of $F$, is univalent for all $a\in\SD$. Letting $|a|\to1^-$ and using Hurwitz' theorem we get that $h+ag$  is univalent for all $a\in\partial\SD$; the fact that $|\om|<1$ guarantees that this limit function is not constant. A standard rotational argument concludes the proof: Let $z_1, z_2$ be distinct points in $\SD$ for which $f(z_1)=f(z_2)$. Since $h$ is injective, we have that $h(z_1)\neq h(z_2)$. Setting $\t=\arg\big(  h(z_1)- h(z_2) \big)$ we see that 
$$
\SR \ni e^{-i\t} \big(h(z_1)- h(z_2)\big) =  e^{-i\t} \big(\overline{g(z_2)}- \overline{g(z_1)}\big) = e^{i\t} \big(g(z_2)- g(z_1)\big), 
$$ 
from which we get that $h+e^{2i\t}g$ is not injective, a contradiction. 

To show that $\|Sh\|\leq 2$, we rewrite \eqref{Sch-HM}, suppressing the variable $z$ where it is possible, as follows
\begin{align*}
Sh  = & \, S_f - \frac{ \overline{\om} \om'}{1-|\om|^2} \left( Ph -\frac{ \overline{\om} \om'}{1-|\om|^2}  - \frac{ 2 \overline{z}}{ 1-|z|^2} \right) \\ 
&+ \frac{ \overline{\om} }{1-|\om|^2}\left[  \om'' - \frac{ 2 \overline{z} \om'}{1-|z|^2} +  \frac{2 \overline{\om} (\om')^2}{1-|\om|^2} \right] - \tfrac{3}{2} \left( \frac{ \overline{\om} \om'}{1-|\om|^2} \right)^2. 
\end{align*}
Multiplying both sides by $(1-|z|^2)^2$, and in view of \eqref{pre-Sch-HM} and \eqref{h-order}, we obtain
\begin{align*}
(1-|z|^2)^2 Sh(z) = & \, (1-|z|^2)^2 S_f(z)  - 2\overline{\om(z)} \om^*(z) A_f(z) - \tfrac{3}{2} \left[ \overline{\om(z)} \om^*(z) \right]^2 \\
&+ 2 \overline{\om(z)} \left[ \frac{(1-|z|^2)^2\om''(z)}{ 2(1-|\om(z)|^2) } - \overline{z} \om^*(z) + \overline{\om(z)} \om^*(z)^2 \right]. 
\end{align*}
Finally, using the hypothesis and Lemma~\ref{lem-om} we get that 
\begin{align*}
(1-|z|^2)^2 |Sh(z)| \leq & \, (1-|z|^2)^2 | S_f(z)|  + 2 |\om^*(z) A_f(z) | + \tfrac{3}{2} \left| \om^*(z) \right|^2  \\
&+ 2 \left| \frac{(1-|z|^2)^2\om''(z)}{ 2(1-|\om(z)|^2) } - \overline{z} \om^*(z) + \overline{\om(z)} \om^*(z)^2 \right| \\ 
\leq &\,  \tfrac{1}{2} \left| \om^*(z) \right|^2 + \tfrac{3}{2} \left| \om^*(z) \right|^2  + 2 \left(1- \left| \om^*(z) \right|^2 \right) \\ 
= & \, 2,
\end{align*}
as desired. 
\end{proof}

Note that for the harmonic M\"obius transformations, \emph{i.e.}, post-compositions of a M\"obius transformation $M$ with an affine map: $f=M+c \overline{M}, c\in\SD$, it holds that $S_f=SM\equiv0$ and $\om^*\equiv0$, so that these mappings satisfy the hypotheses in Theorem~\ref{thm-h-order} and are also, clearly, injective. We do not know if there are any other mappings that satisfy this criterion. It is interesting to ask for simple conditions that imply the inequality in this criterion (see \cite{CE19} for such an analysis for a different criterion). 

Finally, it is fairly simple to combine results from \cite{ACHS26} and \cite{CH07} in order to produce the following. 

\begin{proposition}
Let $f=h+\overline{g}$ be a locally univalent harmonic mapping in $\SD$, for which $h''(0)=0$. If $|Sh(z)|\leq 2c^2$, where $c\approx0.6533$ is the first non-negative solution of the equation $2x\tan x =1$, then $f$ is injective.
\end{proposition}

\begin{proof}
In view of Theorem 5.2 in \cite{ACHS26}, the normalization of $h$ along with the condition $|Sh(z)|\leq 2c^2$ imply that $h$ is convex. Hence, $f$ is injective by \cite{CH07}. 
\end{proof}

\vskip.3cm
%%%%%%%%%%%%%%%%%%%%%%%%%%%%%%%%%%%%%
%%%%%%%%%%%%%%%%%%%%%%%%%%%%%%%%%%%%%
\noindent\emph{Acknowledgements}. The authors wish to thank Mar\'ia Jos\'e Mart\'in for drawing their attention to the non-injective harmonic mapping constructed in \cite{WLRS18} and suggesting its use in the proof of the second inequality in Theorem~\ref{thm-main}.

%%%%%%%%%%%%%%%%%%%%%%%%%%%%%%%%%%%%%
%%%%%%%%%%%%%%%%%%%%%%%%%%%%%%%%%%%%%
\end{document}